\documentclass[preprint,12pt
]{elsarticle}

\usepackage{amssymb,amsmath}

\newcommand{\Real}{\mathbb R}
\newcommand{\p}{\mathcal{P}}
\newcommand{\q}{\mathcal{Q}}
\newcommand{\cv}{\mathcal{V}}
\newcommand{\W}{\mathcal{W}}

\newcommand{\po}{(\mathcal{P}, \mathcal{V})}

\newenvironment{proof}{\noindent\textbf{Proof.}}{\hfill$\Box$}
\newtheorem{thm}{Theorem}[section]

\newtheorem{prop}[thm]{Proposition}
\newtheorem{rem}[thm]{Remark}
\newtheorem{defn}[thm]{Definition}
\newtheorem{lem}[thm]{Lemma}
\begin{document}

\begin{frontmatter}



\title{A note on barreledness in locally convex cones}


\author{Amir Dastouri} 
\ead{a.dastouri@tabrizu.ac.ir}
\author{Asghar Ranjbari\corref{cor1}}
\ead{ranjbari@tabrizu.ac.ir}
\cortext[cor1]{Corresponding author.}
\address{Department of Pure Mathematics, Faculty of Mathematical Sciences,  University of Tabriz, Tabriz, Iran\fnref{label1}
}

\begin{abstract}
Locally convex cones are generalization  of locally convex spaces. The assertion, whether
 a barreled cone is an upper-barreled cone or not, was posed as a question in [A. Ranjbari, H. Saiflu, Projective and inductive limits in locally convex cones, J. Math. Anal. Appl. 332 (2)
(2007) 1097-1108]. In this paper, we show  that a barreled locally convex cone is not necessarily  upper-barreled.
\end{abstract}




\end{frontmatter}


\section{Introduction}
\label{}
A \textit{cone}  is defined to be a commutative monoid $\p$ together with a scalar multiplication
 by nonnegative real numbers satisfying the same axioms as for vector spaces; that is, $\p$ is endowed with an addition $(x,y)\mapsto x+y: \p\times\p \to \p$
 which is associative, commutative and admits a neutral element $0$, i.e. it satisfies:
 $$x+(y+z)=(x+y)+z,$$
 $$x+y=y+x,$$
 $$x+0=x,$$
 for all $x,y,z\in\p$, and with a scalar multiplication $(r,x)\mapsto r\cdot x:\mathbb{R}_+\times\p\to\p$ satisfying for all $x,y\in\p$ and $r,s\in\mathbb{R}_+$:
$$r\cdot(x+y)=r\cdot x+r\cdot y,$$
$$(r+s)\cdot x=r\cdot x+s\cdot x,$$
$$(rs)\cdot x=r\cdot(s\cdot x),$$
$$1\cdot x=x,$$
$$0\cdot x=0.$$
The  theory of locally convex cones  developed in
\cite{kero1992} uses an order theoretical concept or a convex quasi-uniform structure on a  cone. In this paper, we use the former. We  review  some of the
main concepts. For more details refer \cite{kero1992} or \cite{ro2009}, and
for some recent researches see \cite{
ayra2016,ayralattice,ayra2014,ropositivity}.

  A \emph{preordered cone} (\emph{ordered cone}) is a cone $\p$ endowed with a  preorder (reflexive transitive relation) $\leq$
such that addition and multiplication by fixed scalars $r\in\mathbb{R}_+$ are order preserving, that is
$x\le y$ implies $x+z\le y+z$ and $r\cdot x\le r\cdot y$ for all
$x,y,z\in\p$ and $r\in\mathbb{R}_+.$ Every ordered vector space is an ordered cone. The cones $\bar{\Real}=\Real\cup\{+\infty\}$ and $\bar{\Real}_+=\Real_+\cup\{+\infty\}$, with the usual order and
algebraic operations  (especially $0\cdot(+\infty)=0$),   are ordered cones that are not embeddable in vector spaces.

Let  $\p $ be a preordered cone.
A subset  $\cv$ of  $\p $  is called an
\emph{(abstract) 0-neighborhood system}, if $\cv$ is a subcone of  $\p $  without zero directed towards $0$, i.e.:  (i) $0<v$ for all $v\in\cv$; (ii) for all $u,v\in\cv$ there is a  $w\in\cv$ with $w\leq u$ and $w\leq
	v$; (iii) $u+v\in\cv$ and $\alpha v\in\cv$ whenever $u,v\in\cv$ and
	$\alpha>0 $.

Let $a\in\p $  and $v\in\cv$. We define
$v(a)=\{b\in\p \ | \ b\leq a+v\}$,  resp.   $(a)v=\{b\in\p \ | \ a\leq
b+v\},$ to be a neighborhood of $a$ in the  \emph{upper}, resp.
\emph{lower} topologies on $\p$.   The  common refinement of the upper and lower topologies is called
\emph{symmetric} topology. We denote the neighborhoods of $a$ in the
symmetric topology by $v(a)\cap(a)v$ or $v(a)v$.
We call that $(\p,\cv)$ is a
\emph{full locally convex cone} if the elements of $\p$ to be \emph{bounded below}, i.e. for every $a\in\p$ and
$v\in\cv$ we have $0\leq a+\rho v$ for some $\rho>0$.  Each subcone of $\p$, not
necessarily containing $\cv$, is called a \emph{locally convex
	cone}.

For cones $\p$ and $\q$, a mapping $t:\p\to\q$ is
called a \emph{linear operator } if $t(a+b)=t(a)+t(b)$ and $t(\alpha
a)=\alpha t(a)$ hold for $a,b\in\p $ and $\alpha\geq 0$.

A \emph{linear functional} on a cone $\p$ is a linear mapping
$\mu:\p\to\bar{\Real}$.

Let $\po$ and $(\q,\W)$ be two locally convex cones. The linear operator $t:\po\to (\q,\W)$ is called (uniformly) continuous or simply continuous if for every $w\in\W$ one can find a $v\in\cv$ such that $a\le b+v$ implies $t(a)\le t(b)+w$.
It is easy to see that the (uniform)
continuity  implies continuity with respect to the upper, lower and
symmetric topologies on $\p$ and $\q$.

According to the  definition of (uniform) continuity,  a linear functional $\mu$ on $\po$ is (uniformly) continuous if there is a $v\in\cv$ such that
$a\le b+v$ implies $\mu(a)\le\mu(b)+1$.
The continuous linear functionals on a
locally convex cone $(\p,\cv)$  (into $\bar{\Real} $) form a
cone with the usual addition and scalar multiplication of functions.
This cone is called the \emph{dual cone} of $\p$ and denoted by
$\p^*$.

For a locally convex cone $(\p,\cv)$, the  polar $v^{\circ}$  of
$v\in\cv$ consists of all linear functionals $\mu$ on $\p$
satisfying $\mu(a)\leq\mu(b)+1$ whenever $a\leq b+v$ for $a,b\in\p$.
We have $\cup\{v^{\circ}: v\in\cv\}=\p^*$.
The cone $\bar{\Real}_{+}=\{a\in \bar{\Real}:a\geq0 \}$ with (abstract) 0-neighborhood $\cv=\{\varepsilon>0:\varepsilon\in \Real\}$ is a
 locally convex cone. The dual cone of  $\bar{\Real}_{+}$ under
$\cv$ consists of all nonnegative  reals and the functional $\bar{0}$ such that $\bar{0}(a)=0$ for
 all $a\in\Real$ and $\bar{0}(+\infty)=+\infty$.

\section{Some results on barreledness}

Barreledness plays an important role in Functional Analysis to establish some important theorems as Uniform Boundedness  and Open Mapping theorems.
A barrel and a barreled cone were  defined in \cite{ro1998}  for verifying the Uniform Boundedness Theorem in locally convex cones.
In this section, we prove some results  about barreledness  which we will need next.

\begin{defn}[\cite{ro1998}]\label{d1}
	Let $\po$ be a locally convex cone. A \emph{barrel} is a convex
	subset $B$ of $\p^2$ with the following properties:
	\begin{enumerate}
		\item[($B1$)] For every $b\in \p$ one can find a $v\in\cv$  such that for
		every $a\in v(b)v$ there is a $\lambda>0$ such that $(a,b)\in
		\lambda B$.
		\item[($B2$)] For all $a,b\in\p$ such that $(a,b)\notin B$ there is a
		$\mu\in\p^*$ such that $\mu(a)>\mu(b)+1$ and $\mu(c)\leq\mu(d)+1$ for all $(c,d)\in B$.
	\end{enumerate}
\end{defn}

A locally convex cone $\p$ is said to be \emph{barreled} if for
every barrel $B\subseteq\p^2$  and every $b\in \p$ there is a
$v\in\cv$ and a  $\lambda>0$ such that $(a,b)\in\lambda B$ for all
$a\in v(b)v$ (see \cite{ro1998}).

An upper-barreled cone defined in \cite{rasa2007} for verifying inductive and projective limits in
locally convex cone:

\begin{defn}[\cite{rasa2007}]
Let $\po$  be a locally convex cone. The cone $\p$ is called
{\it upper-barreled} if for every barrel $B\subseteq \p^2$, there is
 $v\in\cv$ such that $\tilde{v}\subseteq B$, where
 $$\tilde{v}=\{(a,b)\in \p\times\p \ : \  a\leq b+v\}.$$
\end{defn}

 In \cite{ra2011strict} it was proved that under some conditions, the strict inductive
 limit of barreled locally convex cones is upper-barreled. As mentioned in  \cite{rasa2007}, Example~4.7, a full locally convex cone is upper barreled. Also the cone $\mathcal{C}=\{0,\infty\}$ is upper-barreled with each (abstract) 0-neighborhood system.  An upper-barreled cone is barreled. But
the question whether  a barreled cone is
upper-barreled or not, was posed as a question in \cite{rasa2007} (page 1107). Now, we show that it is not true. First we prove some general results.

\begin{lem}\label{lm2.1}
Let $\po$ be a locally convex cone and $B$ be a barrel in $\p$. If $(a,b)\in \lambda_a B$, $(c,b)\in \lambda_c B$ and $a\le c$ , then $(a,b)\in \lambda_c B$.
\end{lem}

\begin{proof}
By the hypothesis, we have
$$(\frac{a}{\lambda_a},\frac{b}{\lambda_a}),(\frac{c}{\lambda_c},\frac{b}{\lambda_c})\in B.$$
Suppose $a\le  c$ and  $(\frac{a}{\lambda_c},\frac{b}{\lambda_c})\notin B$. By (B2) of Definition~\ref{d1}, there exits $\mu\in \p^*$ such that $\mu (\frac{c}{{\lambda }_{c}})\le \mu (\frac{b}{{\lambda }_{c}})+1 $ and  $\mu (\frac{a}{{\lambda }_{c}})>\mu (\frac{b}{{\lambda }_{c}})+1$. Since $a\le c$, so $\mu(\frac{a}{\lambda_{c}})\leq\mu(\frac{c}{\lambda_{c}})$ and then $\mu (\frac{c}{{\lambda }_{c}})>\mu (\frac{b}{{\lambda }_{c}})+1$. This contradiction shows that $ (\frac{a}{\lambda_c},\frac{b}{\lambda_c} )\in B$.

\end{proof}

\begin{thm}\label{th2.1}
Let $\po$ be a locally convex cone. If  for each $b\in P$ and $v\in V$ there exits $c\in v(b)v$ such that  $a\le c$ for all  $a\in v(b)v$, then $\po$ is barreled.
\end{thm}

\begin{proof}
Let $B$ be a barrel in $\p$ and $b\in \p$. There exits $v\in V$ which satisfies in (B1) and by the  hypothesis, there is  $c\in v(b)v$ such that  $a\le c$ for all  $a\in v(b)v$. By (B1), there is $\lambda_c>0$ such that $(c,b)\in \lambda_c B$. Then the conditions of Lemma~\ref{lm2.1} hold and $\lambda_c>0$ is the same $\lambda>0$ which  we need in the definition of barreled cone.
\end{proof}

\section{A barreled cone which is not upper-barreled}

Now, we construct an example which shows that a barreled locally convex cone is not necessarily upper-barreled.
Consider the set
\begin{equation}\label{example}
\p= \{a_i \ |\  a\in (0,+\infty), i\in \mathbb{N}\}\cup  \{0_0,{\infty }_{\infty } \}.
\end{equation}

\begin{rem}\label{r1}
Note  that $i=0$ if and only if $a_i=0_0$, and $j=\infty$ if and only if $b_j=\infty_{\infty}$.
\end{rem}

We define  addition and  scalar multiplication on $\p$ as follows:\\
\begin{align*}
a_i+b_j:=\left \{
\begin{array}{ccc}
(a+b)_i &&  i=j,\\
a_i    && b_j=0_0,\\
b_j    && a_i=0_0,\\
\infty_\infty && i,j\in\Bbb{N}, \  i\neq j,
\end{array}
\right.
\end{align*}

 $$\lambda \cdot a_i:=(\lambda a)_i,$$  $$0 \cdot a_i=0_0,$$ for all  $a_i,b_j\in \p$ and positive reals $\lambda$.

Also we  consider  the preorder $\preceq$ on $\p$ as follows:

$$a_i\preceq b_j  \Leftrightarrow i=j \ \mbox{and} \  a\leq b.$$

The set $\p$ with the mentioned addition, scalar multiplication and preorder is a preordered cone with $0_0$ as the neutral element.
The set $\cv=\{v\in \Bbb{R} \ | \ v>0\}$ with   the  following property is an (abstract) 0-neighborhood system:\\
\begin{equation}\label{sys}
a_i\preceq b_j +v \ \ \iff \  \  (i=j, \  a \leq b+jv)\ \mbox{or} \  i=0 \ \mbox{ or} \ j=\infty,
\end{equation}
 (see Remark \ref{r1}).

Clearly $\p$ with the neighborhood system $\cv$ is a  locally convex cone.

\begin{prop}\label{p2.1}
In the  locally convex cone $\po$ constructed above, the indices of  members of any symmetric neighborhood of an element is equal to the index of that element.
\end{prop}

\begin{proof}

Let  $v\in \cv$ be arbitrary. It is easy to see that $ v(0_0)=\{0_0\}$ and $(\infty_{\infty})v=\{\infty_{\infty}\}$ and then
 $ v(0_0)v=\{0_0\}$ and $v(\infty_{\infty})v=\{\infty_{\infty}\}$. For $j\neq 0, \infty$ (see Remark~\ref{r1}),
 $v(b_j)=\{a_j | a\leq b+jv\}\cup\{0_0\}$ and $(b_j)v=\{a_j \  |  \ b\leq a+jv\}\cup\{\infty_{\infty}\}$ and then
 $$ v(b_j)v=\{a_j \ | \  a\in [b-jv,+\infty)\cap(0,b+jv]\}.$$
\end{proof}

For each  $j\in \mathbb{N}$, we  set
$$\q_j=  \{b_j\in \p \}\cup \{0_0,{\infty }_{\infty }\}.$$

Clearly, for every $j\in \mathbb{N}$, $\q_j$
 is a subcone of $\p$,
$\p=\cup_{j\in\Bbb{N}}\q_j$ and by Proposition~\ref{p2.1}, $v(b_j)v \subseteq \q_j$ for all $ v\in \cv$ and all $b_j \in \p$.

\begin{rem}\label{rr}
\item[(i)] For every $j\in\Bbb{N}$, $\q_j$ is isomorphic  to $\bar{\Bbb{R}}_+$. Indeed, $\Lambda: \q_j\to\bar{\Bbb{R}}_+, \Lambda (a_i)=\frac{a}{j}$ is a bijective linear (uniformly) continuous monotone order preserving mapping. We note that the inverse of $\Lambda$ is not (uniformly) continuous.
    \item[(ii)] The upper  (and then the symmetric) neighborhoods of $0$ in $\bar{\Bbb{R}}_+$ and $0_0$ in $\q_j$ are different. Indeed, for each $v\in\cv$ we have
        $v(0)=[0,v]$ and $v(0_0)=\{0_0\}$.
\item[(iii)] The locally convex cone $(\bar{\Bbb{R}}_+,\cv)$ is full ($\cv\subset \bar{\Bbb{R}}_+$), but $(\q_j,\cv)$ is not so.
\end{rem}
\begin{lem}
For each $a_j\in \p$ and $\mu\in \p^*$, $\mu(a_j)\ge 0$.
\end{lem}

\begin{proof}
Let $v\in \cv$ such that $\mu \in v^{\circ}$. By the definition of the preorder in $\p$, $0_0\preceq a_j +\lambda v$ for all $\lambda >0 $. So $\mu(0_0)\leq \mu(a_j) +\lambda $ for all $\lambda >0$, then $0\leq \mu(a_j)$.

\end{proof}

For investigation of barreledness in locally convex cones, we need to know duals of cones. 

Let  $\mu\in\q_j^*$ be an arbitrary nonzero element. We have $\mu(0_0)=0$ (since $\mu$ is linear). Also for all $a_j\in\q_j$, we have $a_j=a\cdot 1_j$ (note that since $j\in \Bbb{N}$, then $a_j\neq 0_0 ,\infty_{\infty}$). Then $\mu(a_j)=a\mu(1_j)$. Set $\lambda=\mu(1_j)$. We have $\lambda\in\bar{\Bbb{R}}$ and $\mu(a_j)=\lambda a$. Since $\mu$ is (uniformly) continuous, there exists $v\in\cv$ such that $x\le y+v$ implies $\mu(x)\le \mu(y)+1$ for all $x, y\in\q_j$. We know that  $0_0 \leq a_j+v$ for all $a\in (0,+\infty)$. Then $0\leq\lambda a+1$ for all $a\in \Bbb{R}_+$. This yields that $\lambda \in \bar{\Bbb{R}}_+$. Linearity of $\mu$ implies that $\mu(\infty_{\infty})=0  \  \mbox{or} \  \infty$. On the other hand $a_j\le \infty_\infty+v$ for all $a\in (0,+\infty)$. We conclude that $\lambda a\preceq \mu(\infty_\infty)+1$ for all $a\in (0,+\infty)$. This yields that $\lambda=0$ or $\mu(\infty_\infty)=\infty$. But $\lambda=0$ implies $\mu(\infty_\infty)=\infty$ again, since $\mu$ is nonzero.  Hence the elements of $\q_j^*$ are exactly $\lambda_j$, by
\begin{align}\label{lambda}
\lambda_j(a_i):=\left \{
\begin{array}{ccc}
0 &&  i=0,\\
\lambda a    && i=j,\\
+\infty && i=\infty,
\end{array}
\right.
\end{align}
for all $\lambda\in\Bbb{R}_+$, $\bar 0_j$, by
\begin{align}\label{0bar}
\bar 0_j(a_i):=\left \{
\begin{array}{ccc}
0 &&  i=0\ \mbox{or}\ i=j,\\
+\infty && i=\infty,
\end{array}
\right.
\end{align}
and $\bar\infty$ by
\begin{align}\label{infbar}
\bar \infty(a_i):=\left \{
\begin{array}{ccc}
0 &&  i=0,\\
+\infty && els.
\end{array}
\right.
\end{align}
In fact $\q_j^*=\Bbb{R}_+\cup \{\bar{0_j},\bar{\infty}\}$.

Note that $\bar{\infty}$ as a linear functional is not (uniformly) continuous from $\bar{\Bbb{R}}_+$ to $\bar{\Bbb{R}}_+$. It is not even linear from $\bar{\Bbb{R}}$ to
$\bar{\Bbb{R}}$.
\begin{lem}\label{lq}
Let $(\q_k,\cv)$ be the locally convex cone constructed above. Then
\item[(i)] The polar $v^{\circ}$ contains the functionals $0_k, \bar 0_k, \bar{\infty}$ for each $v\in\cv$.
\item[(ii)] For each $w\in\cv$, $(\frac{1}{w})_k\in(\frac{w}{k})^{\circ}$.
\end{lem}
\begin{proof}
The proof of $(i)$ is clear.
For $(ii)$, let $a_i,b_j\in \q_k$, $a_i\preceq b_j +\frac{w}{k}$ and $\mu=(\frac{1}{w})_k$.
  It is easy to that if $i=0$ or $j=\infty$, then $\mu(a_i)\le\mu(b_j)+1$. Otherwise, we have $i=j=k$. Then $a\le b+k \frac{w}{k}$ and so  $\frac{1}{w}a\le \frac{1}{w}b+1$ i.e. $\mu(a_i)\le\mu(b_j)+1$.
\end{proof}

Now, we investigate the dual of $\p$. Let $\tilde{\mu}: \p\to\bar{\Bbb{R}}$ be a nonzero linear mapping.  Let $i,j\in\Bbb{N}$ and $i\neq j$. So by the definition of the additivity $a_i+b_j=\infty_{\infty}$. We have $\tilde{\mu}(a_i)+\tilde{\mu}(b_j)=\infty$. Then $\tilde{\mu}(a_i)=\infty$ or $\tilde{\mu}(b_j)=\infty$. Suppose $\tilde{\mu}(a_i)<\infty$. Then $\tilde{\mu}(a'_i)<\infty$ for all $a'_i\in \q_i$ (since $i\neq 0,\infty$, so $a'_i=\frac{a'}{a}a_i$). Hence, by the same implication, $\tilde{\mu}(b_j)=\infty$ for all $b_j\in Q_j$ (and so for all $b_j\in \p\setminus\q_i$). This yields that if $\tilde{\mu}\in \p^*$ is nonzero element and $\mu$ is the restriction of $\tilde{\mu}$ on $\q_i$, then  $\tilde{\mu}$ can be uniquely written as follows:
\begin{equation}
{\widetilde{\mu }} (x )= \left\{ \begin{array}{c}
\mu  (x )\ \ \ \ \ \ \ \ \ x\in \q_i, \\
\infty \ \ \ \ \ \ \ \ \ \ \ \ \ \ x\notin \q_i.\  \end{array}\right.
\end{equation}
By (\ref{lambda}), (\ref{0bar}) and (\ref{infbar}), the elements of $\p^*$ are: the linear functional $0$,
\begin{equation}
{\widetilde{\lambda }}_i (x )= \left\{ \begin{array}{c}
\lambda_i (x)\ \ \ \ \ \ \ \ \ \ \ \ x\in \q_i, \\
\infty \ \ \ \ \ \ \ \ \ \ \ \ \ \ x\notin \q_i,\  \end{array}\right.
\end{equation}
for all $\lambda\in\Bbb{R}_+$ and  for all $i\in \Bbb N$,
\begin{equation}
{\widetilde{\bar 0 }}_i (x )= \left\{ \begin{array}{c}
 0 \ \ \ \ \ \ \ \ \ x\in \q_i\setminus \{\infty_{\infty}\}, \\
\infty \ \ \ \ \ \ \ \ x\notin \q_i\setminus \{\infty_{\infty}\},\  \end{array}\right.
\end{equation}
 for all $i\in \Bbb N$
and
\begin{align}\label{infty}
\widetilde{\bar\infty }(x):=\left \{
\begin{array}{ccc}
0 &&  x=0_0,\\
+\infty && els.
\end{array}
\right.
\end{align}

\begin{lem}\label{lp}
Let $(\p,\cv)$ be the locally convex cone constructed in the above. Then
\item[(i)] The polar $v^{\circ}$ contains the functionals $0$, $\tilde0_k, \tilde{\bar 0}_k, \tilde{\bar{\infty}}$ for each $v\in\cv$ and all $k\in \Bbb{N}$.
\item[(ii)] For each $w\in\cv$ and each  $k\in \Bbb{N}$, $\tilde{(\frac{1}{w})}_k\in(\frac{w}{k})^{\circ}$.
\item[(iii)] If $a_i,b_j\in \p$, $a_i\preceq b_j +v$ and $j\notin\{ 0, k\}$, then $\tilde{(\frac{1}{w})}_k(a_i)\le \tilde{(\frac{1}{w})}_k(b_j)+1$ for all $v,w\in\cv$ and all $k\in \Bbb{N}$.
\end{lem}

\begin{proof}
The proofs of $(i)$ and $(ii)$ are similar to Lemma~\ref{lq}.
For $(iii)$, we have $\tilde{(\frac{1}{w})}_k(b_j)=\infty$ and so the proof is complete.
\end{proof}

\begin{prop}
The  cones  $\q_j$ and $\p$ with $\cv$ as an (abstract) 0-neighborhood system are barreled for all $j\in \mathbb{N}
$.
  \end{prop}

  \begin{proof}
 We prove that $\p$ is barreled. Let $B$ be a  barrel in $\p$ and let $b_i\in \p$. For $i=0,\infty$, by Proposition~\ref{p2.1},  all $v\in \cv$ and all $\lambda>0$ satisfy in the definition of barreled cones. Let $j\in\Bbb{N}$. By the definition of the barrel (B1), there exits
  $v\in V$ such that for each $a_j\in v(b_j)v$, there is $\lambda_{a_j}>0$ such that $(a_j, b_j)\in\lambda_{a_j}B$ (see Proposition~\ref{p2.1}). Let $c:= b+jv$. By Theorem~\ref{th2.1}, $\p$ is barreled. For $\q_j$, the proof is similar.
  \end{proof}

\begin{thm}
The locally convex cone $(\p,\cv)$ constructed in the above  is not upper-barreled.
 \end{thm}

\begin{proof}
Let $w\in \cv$ be a fixed element. For each $\ j\in \mathbb{N}$, we define
\begin{equation}\label{Bj}
B_j:=
	\{(a_i,b_k) \ | \ a_i,b_k\in \q_j \  \mbox{and} \ a_i \preceq b_k +\frac{w}{j}\}.
\end{equation}
 We show that $B_j$ is a barrel in $\q_j$:\\
 For (B1), it is enough to consider the same $\frac{w}{j}$ and $\lambda =1$. For (B2), let $ (c_n,d_m)\notin B_j$.
  So $c_n \npreceq d_m +\frac{w}{j}$. By (\ref{sys}), we have two cases: \\
  Case (i):  $n=m=j$ ($c_n, d_m\in\q_j$) and $c> d +w$.
   For  $\mu = {(\frac{1}{w})}_j \in \q_j^*$, we have $\mu(c_n)> \mu(d_m) +1$, and  $\mu(a_i)\le \mu(b_k) +1$ for all $(a_i,b_k)\in B_j$.\\
   Case (ii): $m=0$ and $n\neq 0$. We consider the functional $\mu=\bar{\infty}$. We have $\mu (c_n)=+\infty$ and $\mu (d_m)=0$. So $\mu(c_n)> \mu(d_m) +1$, and  $\mu(a_i)\le \mu(b_k) +1$ for all $(a_i,b_k)\in B_j$.\\

  Now, we define
\begin{equation}\label{B}
B:=\bigcup_{j\in \mathbb{N}}{B_j}.
\end{equation}
We show that $B$ is a barrel in $\p$.

(B1): Let $b_i\in \p$ be arbitrary. Then $b_i\in \q_j$ form some $j\in\Bbb{N}$. Note that $i\in\{0,j,\infty\}$. The set $B_j=B\cap (\q_j\times \q_j)$ is a barrel in $\q_j$. 
 So by considering $ v=\frac{w}{j}\in \cv$ and $ \lambda =1$, $(a_k,b_i )\in \lambda B_j\subseteq \lambda B$ for all $a_k\in v (b_i )v$. Proposition~\ref{p2.1} yields that $k=i$, i.e., $(a_i,b_i )\in \lambda B_j\subseteq \lambda B$ for all $a_i\in v (b_i )v$.

(B2): Let $(a_i,b_j)\notin B$. Then  $a_i\neq0_0$ and $b_j\neq\infty_\infty$ (see (\ref{Bj}) and (\ref{B})).
We consider three cases:\\
Case I: $b_j=0_0$. In this case $a_i\neq 0_0$, since $(0_0,0_0)\in B$. Choose $\mu=\tilde{\bar{\infty}}$. Then  $\mu(a_j)=+\infty$ and then $\mu(a_i)> \mu(b_j) +1$. Also
$\mu(c_m)\le \mu(d_n) +1$ for all $(c_n,d_m)\in B$ by (\ref{Bj}), (\ref{B}) and  Lemma~\ref{lp}, $(i)$.\\
Case II:
$i=j$.
Then $i\neq 0,\infty$ since  $(0_0,0_0),(\infty_\infty,\infty_\infty)\in B$.
By the hypothesis,  $(a_i,b_i)\notin B_i$ and so $a_i\npreceq b_i+\frac{w}{i}$. Then $a>b+w$. By setting $\mu=\tilde{(\frac{1}{w})}_i$, we have $\mu(a_i)> \mu(b_i) +1$ and $\mu(c_m)\le \mu(d_n) +1$ for all $(c_n,d_m)\in B$, by (\ref{Bj}), (\ref{B}) and  Lemma~\ref{lp}, $(ii), (iii)$.\\
Case III: $i\neq j$,  $i,j\neq 0$ and $j\neq\infty$ (note that  $(0_0,b_j), (a_i,\infty_\infty)\in B  $). Choose
$\mu =\tilde{(\frac{1}{w})}_j$. We have $\mu(a_i)> \mu(b_j) +1$ (in fact $\mu(a_i)=\infty$ and $ \mu(b_j)=\frac{b}{w}<\infty$). If  $(c_m,d_n)\in B_j$, $\mu(c_m)\le \mu(d_n) +1$ by (\ref{Bj}), (\ref{B}) and  Lemma~\ref{lp}, $(ii), (iii)$.

Let $u\in \cv$ be arbitrary.  We show that $\tilde{u}\nsubseteq B$. Suppose $j\in \mathbb{N}$ such that $w<ju$, and $a,b\in \mathbb{R}_+$ such that  $w<a-b<ju$. Then
$a_j\preceq b_j +u$ ($a\le b+ ju$) and $a_j\npreceq b_j +\frac{w}{j}$ ($w+b<a$). Hence $(a_j,b_j)\in \tilde{u}$ and $(a_j,b_j)\notin B$.   This shows that $\p$ is not upper-barreled.
\end{proof}




\begin{thebibliography}{99}
\bibitem[ (1)]{ayra2016}
 D. Ayaseh and  A. Ranjbari, { \it Bornological Convergence in Locally Convex Cones},  Mediterr. J.
Math., 13 (4) (2016) 1921-1931.


\bibitem[ (2)]{ayralattice}
 D. Ayaseh and  A. Ranjbari, {\it Locally
  convex quotient lattice cones}, Math. Nachr. 287~(10) (2014) 1083--1092.


 \bibitem[ (3)]{ayra2014}
   D. Ayaseh and  A. Ranjbari, {\it Bornological locally convex cones}, Le Matematiche
   69~(2) (2014) 267--284.


\bibitem[ (4)]{kero1992}
  K. Keimel  and W. Roth,  Ordered cones and approximation, Lecture Notes in Mathematics 1517, Springer-Verlag, Berlin, 1992.



  \bibitem[ (5)]{ra2011strict}
A. Ranjbari,   { \it Strict
  inductive limits in locally convex cones}, Positivity 15~(3) (2011) 465--471.



\bibitem[ (6)]{rasa2007}
  A. Ranjbari and  H. Saiflu,
  {\it Projective and inductive
  limits in locally convex cones}, J. Math. Anal. Appl. 332~(2) (2007)
  1097--1108.



\bibitem[ (7)]{ro2009}
  W. Roth, {\it Operator-valued measures and integrals for cone-valued functions}, Lecture Notes in Mathematics 1964, Springer-Verlag, Berlin, 2009.


\bibitem[ (8)]{ropositivity}
 W. Roth, {\it Korovkin theory for cone-valued functions}, Positivity, 21 (3) (2017) 449--472.



\bibitem[(9)]{ro1998}
  W. Roth, {\it A uniform
  boundedness theorem for locally convex cones}, Proc. Amer. Math. Soc. 126~(7)
  (1998) 1973--1982.
%
%
\end{thebibliography}


\end{document}